\documentclass[11pt,letterpaper]{article}

\usepackage{fullpage}

\usepackage{graphicx,subfigure}
\def\showauthornotes{2}
\def\smallfigs{1}

\def\showkeys{0}
\def\showdraftbox{1}
\def\showcolorlinks{1}
\def\usemicrotype{1}
\def\showfixme{1}

\def\arxivmode{0}
\def\fastmode{0}

\newcommand{\llangle}{\left\langle}
\newcommand{\rrangle}{\right\rangle}

\ifnum\fastmode=0
\usepackage{etex}
\fi

\ifnum\fastmode=0
\usepackage[l2tabu, orthodox]{nag}
\fi

\usepackage{xspace,enumerate}

\usepackage[dvipsnames]{xcolor}

\ifnum\fastmode=0
\usepackage[T1]{fontenc}
\usepackage[full]{textcomp}
\fi

\usepackage[american]{babel}

\usepackage{mathtools}

\usepackage{amsthm}

\newtheorem{theorem}{Theorem}[section]
\newtheorem*{theorem*}{Theorem}

\newtheorem*{proposition*}{Proposition}
\newtheorem{lemma}[theorem]{Lemma}
\newtheorem*{lemma*}{Lemma}
\newtheorem{corollary}[theorem]{Corollary}
\newtheorem*{conjecture*}{Conjecture}

\newtheorem*{fact*}{Fact}

\newtheorem*{exercise*}{Exercise}

\newtheorem*{hypothesis*}{Hypothesis}

\theoremstyle{definition}
\newtheorem{definition}[theorem]{Definition}

\newtheorem{exercise-easy}[theorem]{Exercise}
\newtheorem{exercise-med}[theorem]{Exercise}
\newtheorem{exercise-hard}[theorem]{Exercise$^\star$}

\newtheorem*{claim*}{Claim}

\newtheorem*{remark*}{Remark}
\newtheorem{observation}[theorem]{Observation}
\newtheorem*{observation*}{Observation}

\ifnum\arxivmode=0
\usepackage{newpxtext} %
\usepackage{textcomp} %
\usepackage[varg,bigdelims]{newpxmath}
\usepackage[scr=rsfso]{mathalfa}%
\usepackage{bm} %
\let\mathbb\varmathbb
\fi

\ifnum\arxivmode=1
\usepackage[varg]{pxfonts} %
\usepackage{textcomp} %
\usepackage[scr=rsfso]{mathalfa}%
\usepackage{bm} %
\fi

\ifnum\showkeys=1
\usepackage[color]{showkeys}
\fi

\makeatletter
\@ifpackageloaded{hyperref}
{}
{\usepackage[pagebackref]{hyperref}}

\definecolor{bleudefrance}{rgb}{0.01, 0.1, 1.0}
\definecolor{azure}{rgb}{0.0, 0.5, 1.0}

\ifnum\showcolorlinks=1
\hypersetup{
pagebackref=true,
colorlinks=true,
urlcolor=blue,
linkcolor=blue,
citecolor=OliveGreen}
\fi

\ifnum\showcolorlinks=0
\hypersetup{
pagebackref=true,
colorlinks=false,
pdfborder={0 0 0}}
\fi

\usepackage{prettyref}

\newcommand{\savehyperref}[2]{\texorpdfstring{\hyperref[#1]{#2}}{#2}}

\newrefformat{eq}{\savehyperref{#1}{\textup{(\ref*{#1})}}}
\newrefformat{lem}{\savehyperref{#1}{Lemma~\ref*{#1}}}
\newrefformat{def}{\savehyperref{#1}{Definition~\ref*{#1}}}
\newrefformat{obs}{\savehyperref{#1}{Observation~\ref*{#1}}}
\newrefformat{thm}{\savehyperref{#1}{Theorem~\ref*{#1}}}
\newrefformat{cor}{\savehyperref{#1}{Corollary~\ref*{#1}}}
\newrefformat{cha}{\savehyperref{#1}{Chapter~\ref*{#1}}}
\newrefformat{sec}{\savehyperref{#1}{Section~\ref*{#1}}}
\newrefformat{app}{\savehyperref{#1}{Appendix~\ref*{#1}}}
\newrefformat{tab}{\savehyperref{#1}{Table~\ref*{#1}}}
\newrefformat{fig}{\savehyperref{#1}{Figure~\ref*{#1}}}
\newrefformat{hyp}{\savehyperref{#1}{Hypothesis~\ref*{#1}}}
\newrefformat{alg}{\savehyperref{#1}{Algorithm~\ref*{#1}}}
\newrefformat{rem}{\savehyperref{#1}{Remark~\ref*{#1}}}
\newrefformat{item}{\savehyperref{#1}{Item~\ref*{#1}}}
\newrefformat{step}{\savehyperref{#1}{step~\ref*{#1}}}
\newrefformat{conj}{\savehyperref{#1}{Conjecture~\ref*{#1}}}
\newrefformat{fact}{\savehyperref{#1}{Fact~\ref*{#1}}}
\newrefformat{prop}{\savehyperref{#1}{Proposition~\ref*{#1}}}
\newrefformat{prob}{\savehyperref{#1}{Problem~\ref*{#1}}}
\newrefformat{claim}{\savehyperref{#1}{Claim~\ref*{#1}}}
\newrefformat{relax}{\savehyperref{#1}{Relaxation~\ref*{#1}}}
\newrefformat{red}{\savehyperref{#1}{Reduction~\ref*{#1}}}
\newrefformat{part}{\savehyperref{#1}{Part~\ref*{#1}}}
\newrefformat{ex}{\savehyperref{#1}{Exercise~\ref*{#1}}}
\newrefformat{property}{\savehyperref{#1}{Property~\ref*{#1}}}

\newcommand{\Sref}[1]{\hyperref[#1]{\S\ref*{#1}}}

\usepackage{nicefrac}

\ifnum\fastmode=0
\ifnum\usemicrotype=1
\usepackage{microtype}
\fi
\fi

\ifnum\showauthornotes=2
\newcommand{\mynotes}[1]{{\sffamily\small\color{teal}{#1}}\medskip}
\newcommand{\Authornote}[2]{{\sffamily\small\color{Maroon}{[#1: #2]}}\medskip}
\newcommand{\Authornotecolored}[3]{{\sffamily\small\color{#1}{[#2: #3]}}}
\newcommand{\Authorcomment}[2]{{\sffamily\small\color{gray}{[#1: #2]}}}
\newcommand{\Authorstartcomment}[1]{\sffamily\small\color{gray}[#1: }

\newcommand{\Authorfnote}[2]{\footnote{\color{red}{#1: #2}}}
\newcommand{\Authorfixme}[1]{\Authornote{#1}{\textbf{??}}}
\newcommand{\Authormarginmark}[1]{\marginpar{\textcolor{red}{\fbox{\Large #1:!}}}}
\newcommand{\myexplain}[1]{{\sffamily\small\color{red}{\noindent [Explanation:\medskip\newline \begin{quote}#1\hfill]\end{quote}}}\medskip}
\newcommand{\explain}[1]{{\sffamily\small\color{red}{#1}}\medskip}

\else

\newcommand{\mynotes}[1]{}
\newcommand{\Authornote}[2]{}
\newcommand{\Authornotecolored}[3]{}
\newcommand{\Authorcomment}[2]{}
\newcommand{\Authorstartcomment}[1]{}

\newcommand{\Authorfnote}[2]{}
\newcommand{\Authorfixme}[1]{}
\newcommand{\Authormarginmark}[1]{}
\newcommand{\myexplain}[1]{}
\newcommand{\explain}[1]{}

\ifnum\showauthornotes=1
\renewcommand{\myexplain}[1]{{\sffamily\small\color{red}{\noindent \begin{quote}{\bf Explanation:} \medskip\newline #1\end{quote}}}\medskip}
\fi

\fi

\ifnum\showfixme=0

\fi

\usepackage{boxedminipage}

\newcommand{\tensor}{\otimes}

\newcommand{\textparen}[1]{\text{(#1)}}

\ifx\because\undefined
\newcommand{\because}[1]{\textparen{because #1}}
\else
\renewcommand{\because}[1]{\textparen{because #1}}
\fi

\newcommand{\seteq}{\mathrel{\mathop:}=}

\newcommand\bdot\bullet

\ifx\mathds\undefined %

\else

\fi

\DeclareMathOperator{\vol}{vol}

\DeclareMathOperator{\dist}{dist}

\newcommand{\N}{\mathbb N}
\newcommand{\R}{\mathbb R}

\newcommand{\cC}{\mathcal C}

\newcommand{\cF}{\mathcal F}

\newcommand{\cS}{\mathcal S}
\newcommand{\cT}{\mathcal T}

\newcommand{\bI}{\bm I}

\newcommand{\bS}{\bm S}
\newcommand{\bT}{\bm T}
\newcommand{\bU}{\bm U}

\newcommand{\bbS}{\mathbb S}

\renewcommand{\leq}{\leqslant}

\renewcommand{\geq}{\geqslant}

\ifnum\showdraftbox=1

\else

\fi

\let\epsilon=\varepsilon

\numberwithin{equation}{section}

\newcommand\MYcurrentlabel{xxx}

\newcommand{\MYstore}[2]{%
  \global\expandafter \def \csname MYMEMORY #1 \endcsname{#2}%
}

\newcommand{\MYload}[1]{%
  \csname MYMEMORY #1 \endcsname%
}

\newcommand{\MYnewlabel}[1]{%
  \renewcommand\MYcurrentlabel{#1}%
  \MYoldlabel{#1}%
}

\newcommand{\MYdummylabel}[1]{}

\newcommand{\torestate}[1]{%
  \let\MYoldlabel\label%
  \let\label\MYnewlabel%
  #1%
  \MYstore{\MYcurrentlabel}{#1}%
  \let\label\MYoldlabel%
}

\newcommand{\restatetheorem}[1]{%
  \let\MYoldlabel\label
  \let\label\MYdummylabel
  \begin{theorem*}[Restatement of \prettyref{#1}]
    \MYload{#1}
  \end{theorem*}
  \let\label\MYoldlabel
}

\newcommand{\restatelemma}[1]{%
  \let\MYoldlabel\label
  \let\label\MYdummylabel
  \begin{lemma*}[Restatement of \prettyref{#1}]
    \MYload{#1}
  \end{lemma*}
  \let\label\MYoldlabel
}

\newcommand{\restateprop}[1]{%
  \let\MYoldlabel\label
  \let\label\MYdummylabel
  \begin{proposition*}[Restatement of \prettyref{#1}]
    \MYload{#1}
  \end{proposition*}
  \let\label\MYoldlabel
}

\newcommand{\restatefact}[1]{%
  \let\MYoldlabel\label
  \let\label\MYdummylabel
  \begin{fact*}[Restatement of \prettyref{#1}]
    \MYload{#1}
  \end{fact*}
  \let\label\MYoldlabel
}

\newcommand{\restate}[1]{%
  \let\MYoldlabel\label
  \let\label\MYdummylabel
  \MYload{#1}
  \let\label\MYoldlabel
}

\newcommand{\addreferencesection}{
  \phantomsection
\ifnum\stocmode=0
  \addcontentsline{toc}{section}{References}
\else
  \addcontentsline{toc}{section}{References \hspace*{1in} --------- End of extended abstract ---------}
\fi

}

\newcommand{\e}{\epsilon}

\let\origparagraph\paragraph
\renewcommand{\paragraph}[1]{\vspace*{-3pt}\origparagraph{#1.}}

\allowdisplaybreaks

\usepackage{paralist}

\usepackage{comment}

\let\pref=\prettyref

\newcommand{\diam}{\mathrm{diam}}

\newcommand{\vertiii}[1]{{\left\vert\kern-0.25ex\left\vert\kern-0.25ex\left\vert #1 
          \right\vert\kern-0.25ex\right\vert\kern-0.25ex\right\vert}}

\usepackage[shortlabels]{enumitem}
\usepackage[normalem]{ulem}
\usepackage{stmaryrd}

\newcommand{\cmnt}[1]{}

\newcommand{\len}{\mathrm{len}}
\newcommand{\region}[1]{\left\llbracket #1\right\rrbracket}

\renewcommand{\mathbb}{\vvmathbb}

\newcommand{\tileprod}{\circ}

\newcommand{\figdir}{figures}
\ifnum\smallfigs=1
   \renewcommand{\figdir}{med-figures}
\fi
\ifnum\smallfigs=2
   \renewcommand{\figdir}{small-figures}
\fi

\title{Non-existence of annular separators in geometric graphs}

\author{Farzam Ebrahimnejad\thanks{\texttt{febrahim@cs.washington.edu}}\hspace{0.8in}  James R. Lee\thanks{\texttt{jrl@cs.washington.edu}} \vspace{0.1in}\\
{\small Paul G. Allen School of Computer Science \& Engineering} \\ {\small University of Washington}}
\date{}

\begin{document}

\maketitle

\begin{abstract}
   Benjamini and Papasoglou (2011) showed that planar graphs with uniform polynomial volume growth admit
   $1$-dimensional annular separators: The vertices at graph distance $R$ from any vertex can be separated from those at distance $2R$
   by removing at most $O(R)$ vertices.
   They asked whether geometric $d$-dimensional graphs with
   uniform polynomial volume growth similarly admit $(d-1)$-dimensional annular separators when $d > 2$.
   We show that this fails in a strong sense: For any $d \geq 3$ and every $s \geq 1$, there is a
   collection of interior-disjoint spheres in $\R^d$ whose tangency graph $G$ has uniform polynomial growth,
   but such that all annular separators in $G$ have cardinality at least $R^s$.
\end{abstract}

\section{Introduction}

The well-known Lipton-Tarjan separator theorem \cite{LT79}
asserts that any $n$-vertex planar graph has a balanced separator with $O(\sqrt{n})$ vertices.
By the Koebe-Andreev-Thurston circle packing theorem, every planar graph can be realized
as the tangency graph of interior-disjoint circles in the plane.
One can define $d$-dimensional geometric graphs by analogy:  
Take a collection of ``almost non-overlapping'' bodies $\{S_v \subseteq \R^d : v \in V \}$,
where each $S_v$ is ``almost round,''
and the associated geometric graph contains an edge $\{u,v\}$ if $S_u$ and $S_v$ ``almost touch.''

As a prototypical example, suppose we require that every point $x \in \R^d$ is contained
in at most $k$ of the bodies $\{S_v\}$, each $S_v$ is a Euclidean ball, and two bodies are
considered adjacent whenever $S_u \cap S_v \neq \emptyset$.
These are precisely the intersection graphs of $k$-ply systems of balls, studied 
by Miller, Teng, Thurston, and Vavasis \cite{MTTV97}.
Those authors also provide a generalization of the Lipton-Tarjan separator theorem:
For $k = O(1)$, such an intersection graph contains a balanced separator of size $O(n^{1-1/d})$.

Similarly, finite-element graphs associated to simplicial complexes with bounded aspect ratio
can be viewed as subgraphs of geometric overlap graphs \cite{MTTV98}, and one again obtains
balanced separators of size $O(n^{1-1/d})$.
This covers a number of scenarios commonly arising in applications of the finite-element method;
we refer to the discussion of well-shaped meshes in \cite[\S 6.2]{spielman-teng}.

We will not be too concerned with the particular notion of geometric graph used
since our construction satisfies all these commonly employed sets of assumptions.
Indeed, it can be cast as the tangency graph of a sphere packing, where adjacent spheres
have uniformly comparable radii.  It can also be cast as the $1$-skeleton
of a $d$-dimensional simplicial complex whose simplices have uniformly bounded aspect ratio
as studied.  Such graphs were studied by, for instance, by Plotkin, Rao, and Smith \cite{PRS94} in their work
on shallow minors (see also the followup work \cite{Teng98}).

\medskip
\noindent
{\bf Annular separators.}
Note that the preceding results deal with {\em global} separators that separate
the entire graph into two roughly equal pieces.  In many settings, especially those arising
in physical simulation, it useful to
consider {\em local} separators.  Let $G=(V,E)$ be an undirected graph with path metric $d_G$,
and define graph balls and graph spheres, respectively, by
\begin{align*}
   B_G(x,R) &\seteq \left\{ y \in V : d_G(x,y) \leq R \right\} \\
   S_G(x,R) &\seteq \left\{ y \in V : d_G(x,y) = R \right\}.
\end{align*}

Suppose that for some $C > 1$, we we want to separate $S_G(x,R)$ from $S_G(x,C R)$ by
removing a small set of nodes $U \subseteq B_G(x,C R) \setminus B_G(x,R)$.  We refer to $U$ as an {\em annular separator.}
See \pref{fig:annular-sep}.

\begin{figure}[h]
      \begin{center}
         \subfigure[Separating $S_G(x,R)$ from $S_G(x,2R)$\label{fig:annular_sep}]{\includegraphics[width=7cm]{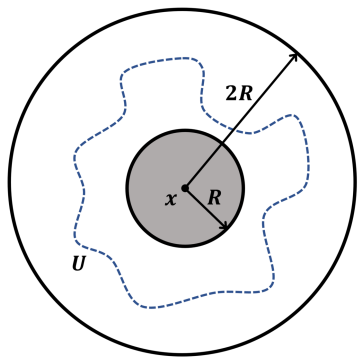}}\hspace{0.6in}
         \subfigure[A random triangulation\protect\footnotemark\ of $\bbS^2$\label{fig:uipt}]{\includegraphics[width=7cm]{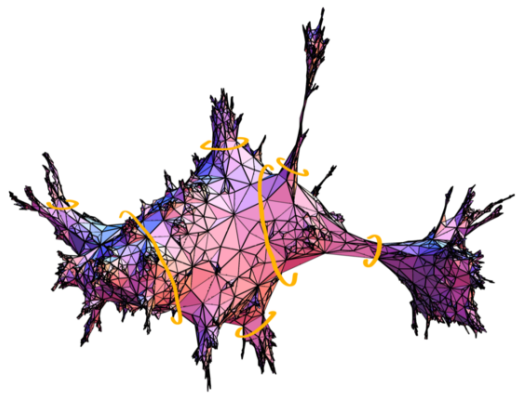}} 
         \caption{Annular separators\label{fig:annular-sep}}
   \end{center}
\end{figure}

\footnotetext{Depiction of a random triangulation is due to Nicolas Curien.}

It is easy to see that even if $G$ is a planar graph, small annular separators don't necessary exist.
But in many cases, one can find annular separators of size $O(R)$.
For instance, this is true for most vertices at most scales
in a uniformly random triangulation of the $2$-dimensional sphere \cite{Krikun05} (a fact
which extends experimentally to a variety of other models of random planar maps, e.g., those studied in \cite{GHS20}).
These models also have the properties that the cardinality of graph balls
tends to grow asymptotically like $|B_G(x,R)| \sim R^{k}$ (as $R \to \infty$, up to lower-order fluctuations),
where the exponent $k$ depends on the model.  For random triangulations, one has $k=4$ \cite{Angel03}.

Benjamini and Papasoglou \cite{BP11} give an explanation of this phenomenon as follows.
Suppose that $G$ is an infinite planar graph and we assume, additionally, that {\em $G$ has uniform polynomial growth of degree $k$:}
There exist numbers $C,k \geq 1$ such that
\[
   C^{-1} R^k \leq |B_G(x,R)| \leq C R^k,\quad \forall x \in V, R \geq 1.
\]
Then for every $x \in V(G)$ and $R \geq 1$, there is a set $U \subseteq B_G(x,2R) \setminus B_G(x,R)$
whose removal disconnects $S_G(x,R)$ from $S_G(x,2R)$ in $G$, and such that $|U| \leq O(R)$.
This applies equally well to finite graphs:  Indeed, the authors actually show that if
the graph metric restricted to $B_G(x,2R)$ has doubling constant $\lambda$, then one can find
an annular separator of size at most $C_{\lambda} R$.

We remark that there are rich families of planar graphs with uniform polynomial growth arising
in a variety of contexts; see \cite{bs01,bk02,EL20}.  Indeed,
one can obtain planar graphs with uniform polynomial growth of
degree $k$ for all real degrees $k > 1$.
Moreover, many models of random planar graphs
have an almost sure asymptotic version of this property \cite{DG20,GHS20}.

The authors of \cite{BP11} asked whether an analog of this phenomenon holds in higher dimensions.
For instance, if $G$ is a graph with uniform polynomial growth that can be geometrically represented
in $\R^3$, does it hold that $G$ has annular separators of size $O(R^2)$?
We give examples showing that for $d \geq 3$, this phenomenon fails in a strong way.

Say that a graph $G$ is {\em sphere-packed in $\R^d$} if $G$ is the tangency
graph of a collection of interior-disjoint spheres in $\R^d$.
Say that $G$ is {\em $M$-uniformly sphere-packed in $\R^d$}
if the collection of spheres can be taken such that the
radii of any two tangent spheres lies in the interval $[M^{-1},M]$, and that
$G$ is {\em uniformly sphere-packed in $\R^d$} if this holds for some $M < \infty$.

\newcommand{\isodim}{\underline{\mathsf{s}}}

\begin{theorem}[Arbitrarily large annular separators]
\label{thm:main1}
For every $d \geq 3$ and $s \geq 1$, there is a number $c > 0$ and a graph $G$ satisfying:
   \begin{enumerate}
      \item $G$ has uniform polynomial growth of degree $O(s)$.
      \item $G$ is uniformly sphere-packed in $\R^d$.
      \item For every $x \in V(G)$, there are at least $c R^s$ vertex-disjoint paths from
            $S_G(x,R)$ to $S_G(x,R')$ for any $R' > R \geq 1$.
   \end{enumerate}
\end{theorem}

Clearly if $G$ has uniform polynomial growth of degree $d$, then one of the 
$R$-many spheres $S(x,R+1),S(x,R+2) \ldots, S(x,2R)$ must be an annular separator of size $O(R^{d-1})$.
We show that the moment the growth degree exceeds $d$,
there are graphs sphere-packed in $\R^d$ that don't have $(d-1)$-dimensional annular separators.

\begin{theorem}[Nearly-dimensional growth rate]
   \label{thm:main2}
   For every $d \geq 3$ and $\e > 0$, there are numbers $c,\delta > 0$ and a graph $G$ satisfying:
   \begin{enumerate}
      \item $G$ has uniform polynomial growth of degree at most $d+\e$.
      \item $G$ is uniformly sphere-packed in $\R^d$.
      \item For every $x \in V(G)$, there are at least $c R^{(d-1)+\delta}$ vertex-disjoint paths from $S_G(x,R)$ to $S_G(x,R')$ for any $R' > R \geq 1$.
   \end{enumerate}
\end{theorem}

Note that the two preceding theorems refer to infinite graphs.  A version for families of 
finite graphs appears in \pref{thm:main-finite}.

\medskip
\noindent
{\bf Preliminaries.}
We will consider primarily connected, undirected graphs $G=(V,E)$,
which we equip with the associated path metric $d_G$.
We write $V(G)$ and $E(G)$, respectively, for the vertex and edge
sets of $G$.
If $U \subseteq V(G)$, we write $G[U]$ for the subgraph induced on $U$.

For $v \in V$, let $\deg_G(v)$ denote the degree of $v$ in $G$.
Let $\diam(G) \seteq \sup_{x,y \in V} d_G(x,y)$ denote the diameter of $G$
(which is only finite for $G$ finite and connected), and
for a subset $S \subseteq V$, denote $\diam_G(S) \seteq \sup_{x,y \in S} d_G(x,y)$.
For $v \in V$ and $r \geq 0$, we use $B_G(v,r) = \{ u \in V : d_G(u,v) \leq r\}$
to denote the closed ball in $G$.
For subsets $S,T \subseteq V$, we write $d_G(S,T) \seteq \inf \{ d_G(s,t) : s \in S, t \in T\}$.

For two expressions $A$ and $B$, we use the notation $A \lesssim B$ to denote
that $A \leq C B$ for some {\em universal} constant $C$.  The notation
$A \lesssim_{\alpha,\beta,\cdots} B$ denotes that $A \leq C(\alpha,\beta,\cdots) B$ where $C(\alpha,\beta,\cdots)$
denotes a number depending only on the parameters $\alpha,\beta,$ etc.
We write $A \asymp B$ for the conjunction $A \lesssim B \wedge B \lesssim A$.

\section{Tilings of the unit cube}

Fix the dimension $d \geq 2$.
Our constructions are based on tilings of subsets of $\R^d$ by axis-parallel hyperrectangles,
a generalization of the planar constructions in \cite{EL20}.
A $d$-dimensional {\em tile} is an axis-parallel closed hyperrectangle $A \subseteq \R^d$,
i.e., a set of the form $A = [a_1,b_1] \times [a_2,b_2] \times \cdots \times [a_d, b_d]$ for
numbers satisfying $a_i < b_i$ for each $i = 1,2,\ldots,d$.

We will encode such a tile as a $(d+1)$-tuple $(p(A), \ell_1(A), \ell_2(A), \ldots, \ell_d(A))$, where
$p(A) \seteq (a_1,a_2,\ldots,a_d)$ and $\ell_i(A) \seteq b_i-a_i$ is the length of the projection
of $A$ along the $i$th axis.

A {\em tiling $\bT$} is a finite collection of interior-disjoint tiles.
Denote $\region{\bT} \seteq \bigcup_{A \in \bT} A$.  If $R \subseteq \R^d$, we say that
{\em $\bT$ is a tiling of $R$} if $\region{\bT} = R$.
We associate to a tiling its {\em tangency graph} $G(\bT)$ with vertex set $\bT$
and with an edge between two tiles $A,B \in \bT$ whenever $A \cap B$ has
non-zero $(d-1)$-dimensional volume.

Denote by $\cT_d$ the set of all tilings of the unit $d$-dimensional cube $[0,1]^d$.
See \pref{fig:tiling} for a tiling of $[0,1]^3$ and \pref{fig:dual_graph} for a representation
of its tangency graph.
For the remainder of the paper, we will consider only tilings $\bT$ for which $G(\bT)$ is connected.

\begin{definition}[Tiling product]
For $\bS, \bT \in \cT_d$, define the product
$\bS\tileprod \bT \in \cT_d$ as the tiling formed by replacing every tile in
$\bS$ by an (appropriately scaled) copy of $\bT$.
More precisely:  For every $A \in \bS$ and $B \in \bT$, there is a tile $R \in \bS\tileprod\bT$
with $\ell_i(R) \seteq \ell_i(A) \ell_i(B)$, and
\[
   p_i(R) \seteq p_i(A)+p_i(B) \ell_i(A),
\]
for each $i=1,2,\ldots,d$.
\end{definition}

If $\bT \in \cT_d$ and $n \geq 0$, we use $\bT^n \seteq \bT \tileprod \cdots \tileprod \bT$ to denote
the $n$-fold tile product of $\bT$ with itself, where $\bT^0 \seteq \bI_d$ and $\bI_d \seteq \{[0,1]^d\} \in \cT_d$
is the identity tiling.
The following observation
shows that this is well-defined.

\begin{figure}[h]
      \begin{center}
         \subfigure[$\bT^{(3)}_{\langle 3,6,3\rangle}$\label{fig:tiling}]{ \includegraphics[width=6cm]{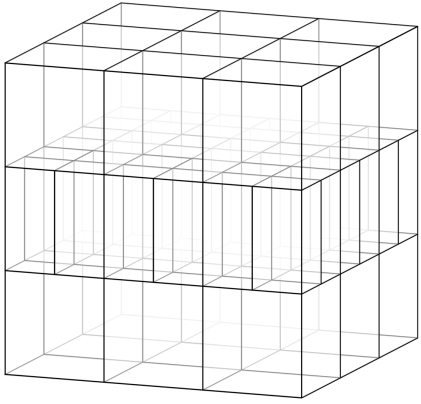}} \hspace{0.3in}
         \subfigure[The tangency graph\label{fig:dual_graph}]{  \includegraphics[width=6.8cm]{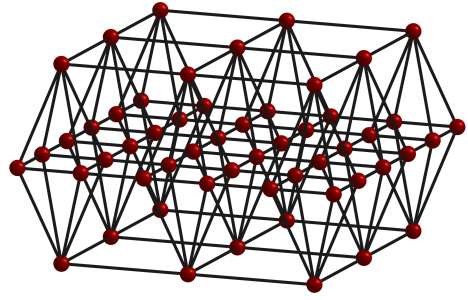}}
         \caption{A tiling and its tangency graph}
   \end{center}
\end{figure}

\begin{observation}
   The tiling product is associative:  $(\bS \tileprod \bT) \tileprod \bU = \bS \tileprod (\bT \tileprod \bU)$
   for all $\bS,\bT,\bU \in \cT_d$.
   Moreover, if $\bm{I_d} \in \cT_d$ consists of the single tile $[0,1]^d$, then $\bT \tileprod \bm{I_d} = \bm{I_d} \tileprod \bT$
   for all $\bT \in \cT_d$.
\end{observation}

\begin{figure}[h]
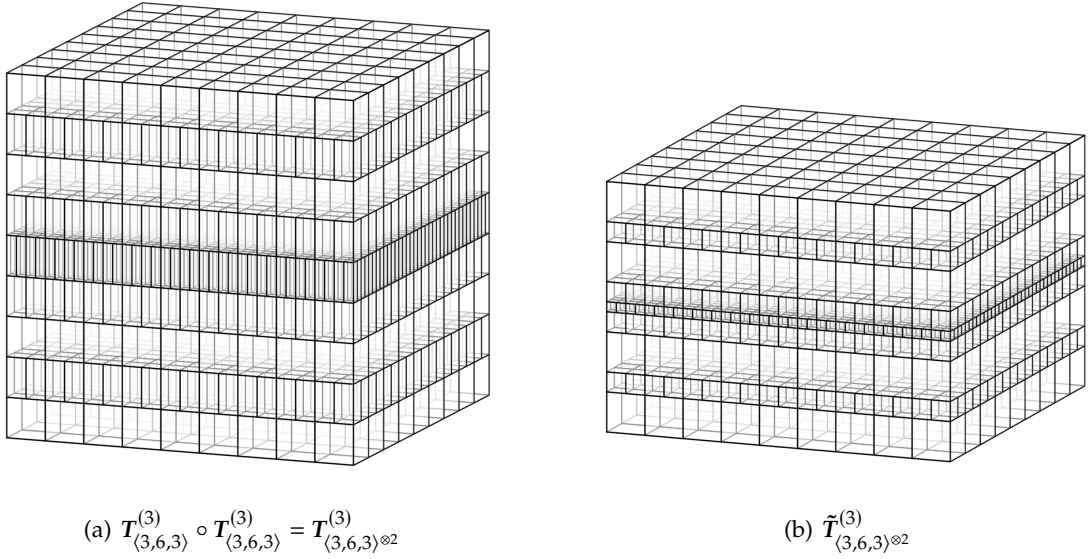

      \begin{center}
         \subfigure[$\bT^{(3)}_{\langle 3,6,3\rangle} \circ \bT^{(3)}_{\langle 3,6,3\rangle} = \bT^{(3)}_{\langle 3,6,3\rangle^{\otimes 2}}$\label{fig:gamma_squared}]{ \includegraphics[width=7cm]{\figdir/tall_cubes_gsquared.png}} \hspace{0.3in}
         \subfigure[$\tilde{\bT}^{(3)}_{\langle 3,6,3\rangle^{\otimes 2}}$\label{fig:cube-packing}]{  \includegraphics[width=7cm]{\figdir/cubes_gsquared.png}}
         \caption{A tile product and its cube packing}
   \end{center}
\end{figure}

\subsection{The construction}

Given a sequence $\gamma = \llangle \gamma_1,\ldots,\gamma_b\rrangle$ with $\gamma_i \in \N$, we define
an associated tiling $\bT^{(d)}_{\gamma} \in \cT_d$ as follows: 
For $1 \leq i \leq b$, fill $[0,1]^{d-1} \times [(i-1)/b,i/b]$
with $\gamma_i^{d-1}$ copies of $[0,1/\gamma_i]^{d-1} \times [0,1/b]$   formed into a $\underbrace{\gamma_i \times \cdots \times \gamma_i}_{d-1 \text{ times}}$
grid.  For example, see $\bT^{(d)}_{\langle 3,6,3\rangle}$ in \pref{fig:tiling}.
The following observation will be useful.

\begin{observation}
\label{obs:tile-prod-sequence}
	For $\gamma = \llangle \gamma_1,\ldots,\gamma_a\rrangle$, $\gamma' = \llangle \gamma'_1,\ldots,\gamma'_{b}\rrangle$, we have 
	$\bT^{(d)}_{\gamma} \tileprod \bT^{(d)}_{\gamma'} = \bT^{(d)}_{\gamma \tensor \gamma'}$
	 where 
	 \[
	 \gamma \tensor \gamma' = \llangle \gamma_1 \gamma'_1, \ldots, \gamma_1 \gamma'_{b}, \ldots, \gamma_a \gamma'_1, \ldots, \gamma_a\gamma'_{b} \rrangle.
	 \]
	In particular, for $n \geq 1$ it holds that $\left(\bT^{(d)}_{\gamma}\right)^n = \bT^{(d)}_{\gamma^{\tensor n}}$, 
   where $\gamma^{\tensor n} \seteq \gamma \otimes \gamma \otimes \cdots \otimes \gamma$ is the $n$-fold tensor product
   of $\gamma$ with itself.  See \pref{fig:gamma_squared} for a representation of
   $\bT^{(3)}_{\langle 3,6,3\rangle^{\otimes 2}}$.
   We will use these two representations interchangeably throughout the paper.
\end{observation}

\newcommand{\gr}[1]{\mathsf{k}^{(#1)}}

We use the notations $|\gamma| \seteq b$ and
\[
   |\gamma|^{(d)} \seteq |\bT_{\gamma}^{(d)}| = \gamma_1^{d-1} + \cdots + \gamma_b^{d-1}.
\]
Note that $|\gamma \otimes \gamma'|^{(d)} = |\gamma|^{(d)} |\gamma'|^{(d)}$,
hence
\begin{equation}\label{eq:iter-size}
   \bT_{\gamma^{\otimes n}}^{(d)} = |\gamma^{(d)}|^n.
\end{equation}
Further, denote 
\[
   \gr{d}(\gamma) \seteq \frac{\log \left(|\gamma|^{(d)}\right)}{\log |\gamma|}.
\]

We now state three key lemmas proved that are proved in subsequent sections,
and then use them to prove our main theorems.
The first establishes uniform polynomial volume growth for the graphs $\bT^{(d)}_{\gamma^{\otimes n}}$.
It is proved in \pref{sec:vol-growth-analysis}.

\begin{lemma}[Volume growth]
\label{lem:T-gamma-vol-growth}
   Consider $\gamma = \llangle \gamma_1,\ldots,\gamma_b\rrangle$ with $\min(\gamma) = \gamma_1 = \gamma_b = b$.
   Then for all $d \geq 2$, there is a number $C=C(d,\gamma)$ such that
   the family of graphs $\cF = \left\{G\left(\bT^{(d)}_{\gamma^{\tensor n}}\right) : n \geq 0\right\}$ has uniform polynomial
   growth of degree $k = \gr{d}(\gamma)$ in the sense that
\[
   C^{-1} R^k \leq |B_G(x,R)| \leq C R^k ,\qquad \forall G \in \cF, x \in V(G), 1 \leq R \leq \diam(G)\,.
\]
\end{lemma}

The second lemma, proved in \pref{sec:sep-size-lower-bound}
establishes a lower bound on the size of annular separators.

\begin{lemma}[Separator size]
\label{lem:T-gamma-separator-size}
For every $d \geq 2$ and $b \geq 1$, there is a number $c=c(d,b)$ such that for
   any $\gamma = \llangle \gamma_1,\ldots,\gamma_b\rrangle$ with $\min(\gamma) = b$, the
   following holds.
   Denote $k \seteq \gr{d-1}(\gamma)$.
   For any $n \geq 1$, if $G = G(\bT^{(d)}_{\gamma^{\tensor n}})$, and $v \in V(G)$,
   then there are at least $c R^k$ disjoint paths from $S_G(v,R)$
   to $S_G(v,R')$, for any $1 \leq R < R' \leq \diam(G)/(3(d+1))$.
\end{lemma}

If $G$ is an undirected graph, let us write $[G]_m$ for the {\em $m$-subdivision of $G$,}
where each edge of $G$ is replaced by a path of length $m$.
We will sometimes consider $V(G) \subseteq V([G]_m)$ via the obvious identification.
The next lemma is proved in \pref{sec:sphere-packing}.

\begin{lemma}
   \label{lem:gamma-sphere-packing}
   Consider a sequence $\gamma = \langle \gamma_1,\ldots,\gamma_b\rangle$ with $\gamma_1 = \gamma_b$,
   and such that
   \begin{equation}\label{eq:integral}
      \max \left\{ \frac{\gamma_{i+1}}{\gamma_i}, \frac{\gamma_i}{\gamma_{i+1}} \right\} \in \N, \quad \forall 1 \leq i < b.
   \end{equation}
   Then for every $d \geq 2$, there are numbers $m=m(d,\gamma)$ and $M=M(d,\gamma)$
   such that for every $n \geq 0$:  If $G = G(\bT^{(d)}_{\gamma^{\otimes n}})$, then $[G]_m$ is $M$-uniformly
   sphere-packed in $\R^d$.
\end{lemma}

With these results in our hand, let us first prove a finitary version of our main theorems.
The corresponding infinite version appears in \pref{sec:infinite-graphs}.
Define
\[
   s(d,k) \seteq d-1 + (k-d)\left(1-\frac{1}{d-1}\right).
\]
This represents our basic tradeoff:  One can construct $d$-dimensional geometric graphs
with uniform polynomial growth of degree $k+\e$ and such that every annular separator
has size $\Omega(R^{s(d,k)})$.

\begin{theorem}[Finite graph families]
   \label{thm:main-finite}
   For every $d \geq 3$, $k \geq d$, and $\e > 0$, there is a family $\cF$ of finite graphs satisfying:
   \begin{enumerate}
      \item For some $\tilde{k} \leq k+\e$ and every $G \in \cF$,
         \[
            |B_G(x,R)| \asymp_{d,\e} R^{\tilde{k}},\quad \forall x \in V(G), 1 \leq R \leq \diam(G).
         \]
      \item Each $G$ is $M$-uniformly sphere-packed in $\R^d$ for some $M \lesssim_{d,\e} 1$.
      \item 
            There is a number $c \gtrsim_{d,\e} 1$
            such that for every $G \in \cF$ and $x \in V(G)$, there are at least $c R^{s(d,k)}$ vertex-disjoint paths from
       $S_G(x,R)$ to $S_G(x,R')$ for every $1 \leq R < R' \leq c \diam(G)$.
   \end{enumerate}
\end{theorem}

In light of \pref{lem:T-gamma-vol-growth}--\pref{lem:gamma-sphere-packing},
we can take $\cF = \left\{ [G(\bT_{\gamma^{\otimes n}}^{(d)})]_m : n \geq 0 \right\}$ as long as we can find
a sequence $\gamma$ suited to the parameters.
To this end, consider $d \geq 2$ and parameters $h,p,q \in \N$ such that $p \geq q d$.
	 Define the sequence
	\[
   \gamma^{(p, q, h)} \seteq \langle b,
   \underbrace{t b, \ldots, t b}_{b-2 \textrm{ copies}}, b\rangle,
\]  
where 
	 $b \seteq h^{q(d-1)}$ and $t \seteq h^{p - d q}$. 
    Note that $t \geq 1$ since $p \geq q d$.
    Moreover, $\gamma^{(p,q,h)}$ satisfies \eqref{eq:integral} by construction.

	\begin{lemma}
	\label{lem:degree-limit} 
	 	Let $d \geq 2$ and $k \seteq p/q$ be as above. Then for all $\epsilon > 0$, there is some $h_0 \in \N$ so that for all $h \geq h_0$ the following statements hold:
	 	\begin{enumerate}
         \item $\left|\gr{d}(\gamma^{(p, q, h)}) - k\right| < \epsilon$.
         \item $\left|\gr{d-1}(\gamma^{(p, q, h)}) - s(d,k)\right| < \epsilon$.		
	 			\end{enumerate}
	 \end{lemma}

\begin{proof}
   It holds that
\begin{align*}
   \gr{d}(\gamma^{(p,q,h)}) = \log_b(|\gamma^{(p,q,h)}|^{(d)}) &= \log_b \left(b^{d-1}(2+(b-2) t^{d-1})\right) \\
   &= \log_b \left(b^{d-1}(2+(b-2)h^{(d-1)(p - d q)})\right) \\
	&= \log_b \left(b^{d-1}(2+(b-2) b^{k - d})\right).
\end{align*}

 Furthermore, we have
\[
\lim_{h \rightarrow \infty} \log_b(|\gamma^{(p,q,h)}|^{(d)}) = 
\lim_{b \rightarrow \infty} \log_b(b^{d-1}(2+(b-2) b^{(k - d)})) = k.
\]
Therefore for all $\epsilon > 0$, 
for all sufficiently large values of $h$, it holds that
$|\gr{d}(\gamma^{(p,q,h)}) - k| \leq \epsilon$.
Similarly, we have
\[
   \lim_{b \to \infty} \log_b (|\gamma^{(p,q,h)}|^{(d-1)}) = s(d,k),
\]
hence we can choose $h$ sufficiently large to as to satisfy the second condition as well.
\end{proof}

\subsection{Construction of the infinite graphs}
\label{sec:infinite-graphs}

For this section alone, we consider tilings of the nonnegative orthant $[0,\infty)^d$.

Note that for any sequence $\gamma = \langle \gamma_1,\ldots,\gamma_b\rangle$,
one can view $G(\bT_{\gamma}^{(d)})$ as the tangency graph
of a packing of cubes by changing the height of cubes in layer $i$
from $1/b$ to $b/\gamma_i$.  See \pref{fig:cube-packing} for an example.
Let $\tilde{\bT}_{\gamma}^{(d)}$ denote this rescaled tiling.  By convention, we insist that
one corner of the tiling still lies at the origin, implying that
\[
   \llbracket \tilde{\bT}_{\gamma}^{(d)} \rrbracket = [0,1]^{d-1} \times [0,H(\gamma)],
\]
where
\[
   H(\gamma) \seteq b \left(\frac{1}{\gamma_1} + \cdots + \frac{1}{\gamma_b}\right).
\]
If we now assume that $\gamma_1=\gamma_b=b$,
then we have the chain of inclusions:
\[
   \tilde{\bT}^{(d)}_{\gamma^{\tensor 0}} \subseteq \cdots \subseteq  \tilde{\bT}^{(d)}_{\gamma^{\tensor (n-1)}} \subseteq  \tilde{\bT}^{(d)}_{\gamma^{\tensor n}} \subseteq \tilde{\bT}^{(d)}_{\gamma^{\tensor (n+1)}} \subseteq \cdots
\]
This gives rise, in a straightforward way, to the infinite tiling $\tilde{\bT}^{(d)}_{\gamma^{\otimes \N}}$, with
$\llbracket \tilde{\bT}^{(d)}_{\gamma^{\otimes \N}}\rrbracket = [0,\infty)^d$.
We define the infinite tangency graph $\hat{G}_{\gamma}^{(d)} \seteq G({\bT}_{\gamma^{\tensor \N}}^{(d)})$.

\begin{theorem}
   If $\gamma = \langle \gamma_1,\ldots,\gamma_b\rangle$ has $\gamma_1=\gamma_b=b=\min(\gamma)$, then
   $\hat{G}_{\gamma}^{(d)}$ has uniform polynomial growth of degree $\gr{d}(\gamma)$, and
   for every $x \in V(G)$ and $R' > R \geq 1$, there are at least $c R^{\gr{d-1}(\gamma)}$
   vertex-disjoint paths from $S_G(x,R)$ to $S_G(x,R')$, where $c \gtrsim_{\gamma,d} 1$.
\end{theorem}

\begin{proof}
   For $n \geq 1$, denote $G_n \seteq G(\bT_{{\gamma}^{\tensor n}}^{(d)})$ and $\hat{G} \seteq \hat{G}_{\gamma}^{(d)}$.
   We can think of $G_n$ as an induced subgraph of $\hat{G}$ in the obvious way.
   Consider a vertex $v \in V(\hat{G})$ and radii $R' > R \geq 1$.
   Then there is some $n \geq 1$ such that $B_{\hat{G}}(v,2 R') \subseteq V(G_n)$.
   
   In this case, the it holds that $d_{\hat{G}}(x,y) = d_{G_n}(x,y)$ for all $x,y \in B_{\hat{G}}(v,R')$,
   since any path originating in $B_{\hat{G}}(v,R')$ and leaving $V(G_n)$ must have length at least $2R'$.
   Therefore $B_{\hat{G}}(v,R')=B_{G_n}(v,R')$, and the uniform polynomial growth and seperator size
   assertions then follow immediately from \pref{lem:T-gamma-vol-growth} and \pref{lem:T-gamma-separator-size}.
\end{proof}

The following theorem is proved in \pref{sec:sphere-packing}.

\begin{theorem}\label{thm:sphere-packing-ghat}
   If $\gamma=\langle \gamma_1,\ldots,\gamma_b \rangle$ satisfies $\gamma_1=\gamma_b$ and \eqref{eq:integral},
   then there is some number $m=m(d,\gamma)$ such that $[\hat{G}_{\gamma}^{(d)}]_m$ is uniformly sphere-packed in $\R^d$.
\end{theorem}

\section{Volume growth analysis}
\label{sec:vol-growth-analysis}

Our goal is now to prove \pref{lem:T-gamma-vol-growth}. 
The next section provides a few key lemmas about the size of balls in products of tilings,
which are mostly straightforward generalizations of the bounds in \cite{EL20} (for the case $d=2$).
With these in hand, we prove \pref{lem:T-gamma-vol-growth} in \pref{sec:vol-growth-gamma}. 

\subsection{Volume growth in tile products}
\label{sec:vol-growth-prod}

The next lemma is straightforward.

\begin{lemma}
\label{lem:contract_growth_lower_bound}
Consider $\bS,\bT \in \cT_d$ and $G=G(\bS\tileprod\bT)$.
If $X \in \bS\tileprod\bT$, then $|B_G(X, \diam(G(\bT)))| \geq |\bT|$.
\end{lemma}

Let $\{e_1,\ldots,e_d\}$ denote the standard basis of $\R^d$.
If $\bT \in \cT_d$, we write $E_i(\bT) \subseteq E(G(\bT))$ for the set of edges in the $i$th direction,
i.e., those edges $\{A,B\} \in E(G(\bT))$ where $A \cap B$ is orthogonal to $e_i$.
Thus we have a partition $E(G(\bT)) = E_1(\bT) \cup \cdots \cup E_d(\bT)$.

For $A \in \bT$ and $1 \leq i \leq d$, 
denote
\[
   N_{\bT}(A,i) = \{ A' \in \bT : \{A,A'\} \in E_i(\bT) \},
\] 
and $N_{\bT}(A) \seteq N_{\bT}(A,1) \cup \cdots \cup N_{\bT}(A, d)$.
Moreover, we define
\begin{align*}
\alpha_{\bT}(A,i) &\seteq \max \left\{ \frac{\ell_j(A)}{\ell_j(B)}: B \in N_{\bT}(A, i) , 1 \leq j \leq d \right\}\\
   \alpha_{\bT}(i) &\seteq \max \left\{ \alpha_{\bT}(A, i) : A \in \bT\right\} \\
      \alpha_{\bT} &\seteq \max \left\{ \alpha_{\bT}(i) : 1 \leq i \leq d \right\} \\
   L_{\bT} &\seteq \max \left\{ \ell_i(A) : A \in \bT, 1 \leq i \leq d\right\}.
\end{align*}

We will take $\alpha_{\bT} \seteq 1$ if $\bT$ contains a single tile.
It is now straightforward to check that $\alpha_{\bT}$ bounds the degrees in $G(\bT)$.
\begin{lemma}
\label{lem:degree_bound}
   For a tiling $\bT \in \cT_d$ and $A \in \bT$, it holds that
   \[
      \deg_{G(\bT)}(A) \leq 3^d\cdot \alpha_{\bT}^{2d}.
   \]
\end{lemma}

\begin{proof}
Denote
\[
\tilde{A} \seteq \left\{(x_1, \ldots, x_d) \in \R^d:  p_i(A) - \alpha_{\bT} \ell_i(A) \leq x_i \leq p_i(A) + (1 + \alpha_{\bT}) \ell_i(A) \right\},
\] 
where $p_i(A)$ denotes the $i$th coordinate of $p(A)$.
Clearly 
\[
   \vol_d(\tilde{A}) = \prod_{i=1}^d \left((1+2\alpha_{\bT})\ell_i(A)\right) = (1+2\alpha_{\bT})^d \vol_d(A).
\]
Furthermore, by the definition of $\alpha_{\bT}$, it holds that $A' \subseteq \tilde{A}$ for $A' \in N_{\bT}(A)$.
And for any such $A'$, it holds that $\vol_d(A') \geq \alpha_{\bT}^{-d} \vol_d(A)$.  Hence,
\[
   |N_{\bT}(A)| \alpha_{\bT}^{-d} \vol_d(A) \leq (1+2\alpha_{\bT})^d \vol_d(A).
\]
Using $\alpha_{\bT} \geq 1$, this yields $\deg_{G(\bT)}(A) \leq 3^d \alpha_{\bT}^{2d}$.
\end{proof}

\begin{lemma}
\label{lem:growth_upper_bound}
Consider $\bS,\bT \in \cT_d$ and let $G = G(\bS \tileprod \bT)$.
Then for any $X \in \bS \tileprod \bT$, it holds that
\begin{equation}\label{eq:growth}
   |B_G(X,1/(\alpha^{2d}_{\bS} L_{\bT}))| \leq (3\alpha_{\bS}^2)^{d^2} (d+1)|\bT|.
\end{equation}
\end{lemma}

\begin{proof}
   For $Y \in \bS \tileprod \bT$, let $\hat{Y} \in \bS$ denote the unique tile for which $Y \subseteq \hat{Y}$.
   Let $\Gamma_{\bS}(\hat{X})$ denote the set of paths $\gamma$ in $G(\bS)$ that originate from $\hat{X}$
   and such that $|\gamma \cap E_i(\bS)| \leq 1$ for each $i=1,\ldots,d$.
   In other words, the set of paths in $G(\bS)$ starting at $\hat{X}$ and containing at most one edge
   in every direction.
   Denote by $\tilde{N}_{\bS}(\hat{X}) \seteq \bigcup_{\gamma \in \Gamma_{\bS}(\hat{X})} V(\gamma)$
   the set of vertices reachable via such paths.

   Note that because we allow one edge in every direction,
   \begin{equation}\label{eq:minkowski}
     \hat{X} + [-\ell_1(\hat{X})/\alpha^d_{\bS}, \ell_1(\hat{X})/\alpha^d_{\bS}] \times \cdots \times [-\ell_d(\hat{X})/\alpha^d_{\bS}, \ell_d(\hat{X})/\alpha^d_{\bS}] \subseteq \region{\tilde{N}_{\bS}(\hat{X})},
  \end{equation}
   where '$+$' here is the Minkowski sum $R+S \seteq \{ r + s : r \in R, s \in S \}$.

   We will now show that
   \begin{equation}\label{eq:regions}
      \region{B_G(X,1/(\alpha^{2d}_{\bS} L_{\bT}))}  \subseteq \region{\tilde{N}_{\bS}(\hat{X})}.
   \end{equation}
   It will follow that 
   \[
      |B_G(X,1/(\alpha^{d+1}_{\bS} L_{\bT}))| \leq |\bT| \cdot |\tilde{N}_{\bS}(\hat{X})| \leq |\bT| \cdot (d+1)\left(\max_{A \in \bS} \deg_{G(\bS)}(A)\right)^d,
   \]
   and then \eqref{eq:growth} follows from \pref{lem:degree_bound}.

   To establish \eqref{eq:regions}, consider any path $\llangle X=X_0,X_1,X_2,\ldots,X_h\rrangle$ in $G$
   with $\hat{X}_h \notin \tilde{N}_{\bS}(\hat{X})$.
   Let $k \leq h$ be the smallest index for which $\hat{X}_k \notin \tilde{N}_{\bS}(\hat{X})$.
   Then:
   \begin{align}
      &{X_0},{X_1},\ldots,{X_{k-1}} \subseteq \region{\tilde N_{\bS}(\hat{X})}\label{eq:gp1} \\
      &X_{k-1} \cap \left(\partial \region{\tilde N_{\bS}(\hat{X})} \cap (0,1)^d\right) \neq \emptyset.\label{eq:gp2}
   \end{align}
   Now \eqref{eq:gp1} implies that for $j \leq k-1$, we have $\hat{X}_j \in \tilde{N}_{\bS}(\hat{X})$, which implies
   $\ell_i(\hat{X}_j) \leq \alpha^d_{\bS} \ell_i(\hat{X})$ since $\hat{X}_j$ can be reached
   from $\hat{X}$ by a path of length at most $d$ in $G(\bS)$.  It follows that
   \begin{equation}\label{eq:gp3}
      \ell_i(X_j) \leq L_{\bT} \ell_i(\hat{X}_j) \leq L_{\bT} \alpha^d_{\bS} \ell_i(\hat{X}),\qquad j \leq k-1, 1 \leq i \leq d.
   \end{equation}
   And \eqref{eq:gp2} together with \eqref{eq:minkowski} shows that
   \begin{equation}\label{eq:gp4}
      \sum_{j=0}^{k-1} \ell_i(X_j) \geq \ell_i(\hat{X})/\alpha^q_{\bS}, \qquad 1 \leq i \leq d.
   \end{equation}
   Combining \eqref{eq:gp3} and \eqref{eq:gp4} now gives
   \[
      h-1 \geq k-1 \geq \frac{1}{\alpha_{\bS}^{2d} L_{\bT}},
   \]
   verifying \eqref{eq:regions} and completing the proof.
\end{proof}

\subsection{Volume growth in iterated products}
\label{sec:vol-growth-gamma}

Our goal is now to prove \pref{lem:T-gamma-vol-growth}.  To this end, fix $d \geq 2$.

\begin{observation}
\label{obs:gamma-alpha-bound}
   For $1 \leq i \leq d-1$, we have $\alpha_{\bT^{(d)}_{\gamma}}(i) = 1$. It further holds that
	\[
      \alpha_{\bT^{(d)}_{\gamma}} = \alpha_{\bT^{(d)}_{\gamma}}(1) = \max
      \left\{ \frac{\gamma_i}{\gamma_{j}} : 1 \leq i,j \leq b, |i-j|=1 \right\}.
	\]
\end{observation}

Given \pref{obs:tile-prod-sequence} and \pref{obs:gamma-alpha-bound}, the next lemma is straightforward.

\begin{lemma}
Consider $\gamma = \llangle \gamma_1,\ldots,\gamma_b\rrangle$ and $\gamma' = \llangle \gamma'_1,\ldots,\gamma_{b'}\rrangle$.
If $\gamma'_1 = \gamma'_{b'}$, then
\[
\alpha_{\bT^{(d)}_{\gamma}\tileprod \bT^{(d)}_{\gamma'}} = \alpha_{\bT^{(d)}_{\gamma \tensor \gamma'}} \leq \max(\alpha_{\bT^{(d)}_{\gamma}}, \alpha_{\bT^{(d)}_{\gamma'}}).
\] 
\end{lemma}

\begin{proof}
   Note that, by definition for $1 \leq j \leq b'$ and $0 \leq i \leq b-1$, we have
      $(\gamma \otimes \gamma')_{i b'+j} = \gamma_i \gamma'_j$.
      Hence,
   \[
      \frac{(\gamma \otimes \gamma')_{i b'+j+1}}{(\gamma \otimes \gamma')_{i b'+j}} = 
      \begin{cases}
         \frac{\gamma'_{j+1}}{\gamma'_j} & j < b' \\
         \frac{\gamma_{i+1}}{\gamma_i} \frac{\gamma'_1}{\gamma'_{b'}} = \frac{\gamma_{i+1}}{\gamma_i} & j = b',
      \end{cases}
   \]
   where we used $\gamma'_1=\gamma'_{b'}$ in the second case.
\end{proof}

\begin{corollary}
\label{cor:alpha-power-bound}
	Let $\gamma = \llangle \gamma_1,\ldots,\gamma_b\rrangle$. If $\gamma_1 = \gamma_b$, then for $n \geq 1$ we have
\[
\alpha_{\bT^{(d)}_{\gamma^{\tensor n}}}
 \leq \alpha_{\bT^{(d)}_{\gamma}}.
\]
		\end{corollary}

\begin{lemma}
\label{lem:HabnDiam}
Consider $\gamma = \langle \gamma_1,\ldots,\gamma_b\rangle$ with $\gamma_1=\gamma_b=b$.
For every $n \geq 0$, it holds that
\begin{equation}\label{eq:diam}
   b^n-1 \leq \diam\left(G(\bT_{\gamma^{\tensor n}}^{(d)})\right) \leq (d+1) b^n.
\end{equation}
\end{lemma}

\begin{proof}
   Let $\mathcal{C}$ denote the ``bottom'' layer of tiles in $\bT^{(d)}_{\gamma^{\tensor n}}$,
   i.e., those contained in $[0,1]^{d-1} \times [0,1/b^n]$. 
   Take $G \seteq G(\bT_{\gamma^{\tensor n}}^{(d)})$.
    For every $A \in \bT^{(d)}_{\gamma^{\tensor n}}$, it holds that $\ell_d(A) = b^{-n}$.
      In particular, this yields the lower bound in \eqref{eq:diam} since any path
      from the bottom to the top (in dimension $d$) requires $b^n$ tiles.
    
      This also implies that for $A \in \bT^{(d)}_{\gamma^{\tensor n}}$ we have $d_G(A, \mathcal{C}) \leq b^n$. 
     Furthermore, as $\gamma_1 = b$ by assumption,
     we have $\left(\gamma^{\otimes n}\right)_1 = b^n$, and therefore by construction $G[\mathcal{C}]$ is a
      $(d-1)$-dimensional $b^n \times b^n \times \cdots \times b^n$ grid.
      Hence we have $\diam_G(\mathcal{C}) \leq (d-1)b^n$, and now the 
      upper bound in \pref{eq:diam} follows by the triangle inequality.
\end{proof}

\begin{proof}[Proof of \pref{lem:T-gamma-vol-growth}]
   Denote $G \seteq G(\bT_{\gamma^{\tensor n}}^{(d)})$ and fix $A \in \bT_{\gamma^{\tensor n}}^{(d)}$.
   For any $0 \leq j \leq n$,
	write $\bT_{\gamma^{\tensor n}}^{(d)} = \bT_{\gamma^{\tensor (n-j)}}^{(d)} \tileprod \bT_{\gamma^{\tensor j}}^{(d)}$
   and let $\hat{A}$ denote the copy of $\bT_{\gamma^{\tensor j}}^{(d)}$ containing $A$.
   By \pref{lem:HabnDiam}, we have $\diam(\hat{A}) \leq (d+1)b^j$, and thus
   $B_G(A, (d+1) b^j) \supseteq \hat{A}$.

   Employing \pref{lem:contract_growth_lower_bound} therefore gives
   \[
      \left|B_{G}(A, (d+1) b^j)\right| \geq |\hat{A}| = \left|\bT_{\gamma^{\tensor j}}^{(d)}\right| = \left(|\gamma|^{(d)}\right)^j,\qquad \forall A \in \bT_{\gamma^{\tensor n}}^{(d)},\ j \in \{0,1,\ldots,n\}.
   \]
   The desired lower bound now follows using monotonicity of $|B_{G}(A,r)|$ with respect to $r$.

   To prove the upper bound, first note that by
     \pref{cor:alpha-power-bound}
we have
   \[
   \alpha_{\bT_{\gamma^{\tensor n}}^{(d)}} \leq \alpha_{\bT_{\gamma}^{(d)}}  \lesssim_{\gamma} 1.
   \]
   Moreover, as $\min(\gamma) = b$, we have $L_{\bT_{\gamma^{\otimes j}}^{(d)}} = b^{-j}$, hence applying \pref{lem:growth_upper_bound} 
with $S=\bT^{(d)}_{\gamma^{\otimes (n-j)}}$ and $T=\bT^{(d)}_{\gamma^{\otimes j}}$ gives
   \[
      \left|B_{G}(A, b^{j}/\xi)\right| \lesssim_{d, \gamma} |\bT^{(d)}_{\gamma^{\otimes j}}| = \left(|\gamma|^{(d)}\right)^j,\qquad \forall A \in \bT_{\gamma^{\tensor n}}^{(d)},\ j \in \{0,1,\ldots,n\},
   \]
   for some $\xi \lesssim_{\gamma, d} 1$,
   completing the proof.
\end{proof}

\section{The size of annular separators}
\label{sec:sep-size-lower-bound}

We now prove that the graphs $G(\bT_{\gamma^{\tensor n}}^{(d)})$ do not have small annular separators.

\begin{definition}[Projection of tiles]
	For $d \geq 2$ and a tile $A \subseteq \R^d$,
   we write $\Pi_{d-1}(A) \subseteq \R^{d-1}$ for the projection of $A$ onto the last $d-1$ coordinates.
   Furthermore, for a tiling $\bT \in \cT_d$, we define
	 \[
       \Pi_{d-1}(\bT) \seteq \{ \Pi_{d-1}(A) \mid A \in \bT\},
	 \]
    and for a tile $B \in \Pi_{d-1}(\bT)$,
	 \[
       \Pi_{d-1}^{-1}(B;\bT) \seteq \{A \in \bT : \Pi_{d-1}(A) = B\}.
	 \]
\end{definition}

\begin{observation}
\label{obs:construction-projection}
	For any sequence $\gamma = \llangle \gamma_1, \ldots, \gamma_b \rrangle$, the following hold:
   \begin{enumerate}[(a)]
      \item \label{item:cp1} $\Pi_{d-1}(\bT_{\gamma}^{(d)}) = \bT_{\gamma}^{(d-1)}$.
      \item \label{item:cp2} For all $B \in \Pi_{d-1}(\bT_{\gamma}^{(d)})$, the tangency graph $G(\Pi_{d-1}^{-1}(B;\bT_{\gamma}^{(d)}))$ is a path.
      \item For all $B \neq B' \in \Pi_{d-1}(\bT_{\gamma}^{(d)})$, the sets $\Pi_{d-1}^{-1}(B;\bT_{\gamma}^{(d)})$ and
         $\Pi_{d-1}^{-1}(B';\bT_{\gamma}^{(d)})$ are disjoint.
   \end{enumerate}
See \pref{fig:projection}.
\end{observation}

We will now use the family $\{ \Pi_{d-1}^{-1}(B;\bT_{\gamma}^{(d)}) : B \in \Pi_{d-1}(\bT_{\gamma}^{(d)})\}$
of pairwise disjoint paths to locate paths across annuli in $G(\bT_{\gamma^{\tensor n}}^{(d)})$.
We first remark that diameter of each such path is long in $G$.

\begin{lemma}\label{lem:inverse-diam}
   For every $B \in \Pi_{d-1}(\bT_{\gamma}^{(d)})$, 
   \[
      \diam_{G(\bT^{(d)}_{\gamma})}(\Pi_{d-1}^{-1}(B;\bT_{\gamma}^{(d)})) \geq \frac{1}{L_{\bT_{\gamma}^{(d)}}} - 1,
   \]
\end{lemma}

\begin{proof}
   Let $\Pi_1 : \R^d \to \R$ denote the projection onto the first coordinate.
   Then the Euclidean diameter of $\Pi_1(\Pi_{d-1}^{-1}(B;\bT_{\gamma}^{(d)}))$ is precisely $1$.
   On the other hand, for any path $P$ in $G(\bT_{\gamma}^{(d)})$, the Euclidean diameter
   of $\Pi_1(P)$ is at most $(\len(P)+1) L_{\bT_{\gamma}^{(d)}}$, and the result follows.
\end{proof}

\begin{figure}[h]
      \begin{center}
         \includegraphics[width=10cm]{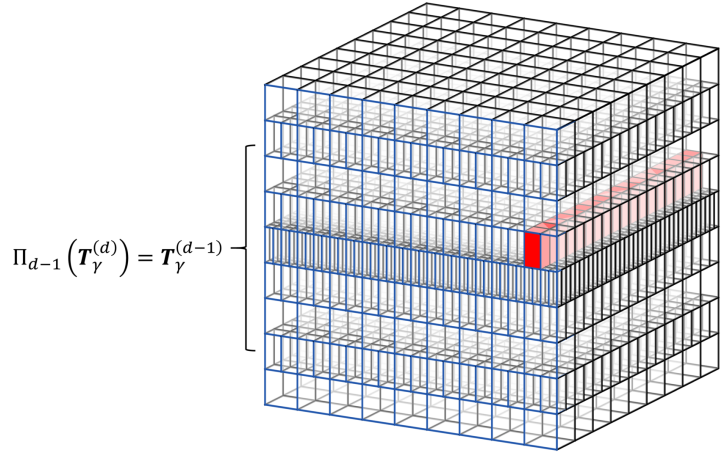}
         \caption{Projection to the last $d-1$ coordinates and $\Pi_{d-1}^{-1}(B;\bT^{(d)}_{\gamma})$ for one $B \in \Pi_{d-1}(\bT^{(d)}_{\gamma})$.\label{fig:projection}}
   \end{center}
\end{figure}

With this in hand, we can now exhibit many disjoint paths across annuli in $G$.

\begin{proof}[Proof of \pref{lem:T-gamma-separator-size}]
   Define $G \seteq G(\bT_{\gamma^{\tensor n}}^{(d)})$ and $h \seteq \lfloor \log_b(R/(d+1)) \rfloor$.
   Let $A \in \bT_{\gamma^{\tensor n}}^{(d)}$ be arbitrary.
   Write $\bT_{\gamma^{\tensor n}}^{(d)} = \bT^{(d)}_{\gamma^{\tensor (n-h)}} \circ \bT^{(d)}_{\gamma^{\tensor h}}$,
   and let $\hat{A} \in \bT^{(d)}_{\gamma^{\tensor n}}$ be the copy of $\bT^{(d)}_{\gamma^{\tensor h}}$ that contains $A$.
   By \pref{lem:HabnDiam}, we have $\diam_G(\hat{A}) \leq (d+1) b^h \leq R$, where the latter inequality
   follows from our definition of $h$.  Thus $\hat{A} \subseteq B_G(A,R)$.

   But as $\hat{A}$ is a translation of $\bT^{(d)}_{\gamma^{\tensor h}}$, by \pref{obs:construction-projection}(a),
   it holds that $\Pi_{d-1}(\bT^{(d)}_{\gamma^{\tensor h}})$ is a translation of $\bT^{(d-1)}_{\gamma^{\tensor h}}$.
   This yields
   \[
      |\Pi_{d-1}\left(B_G(A,R)\right)| \geq |\Pi_{d-1}(\hat{A})| = |\bT^{(d-1)}_{\gamma^{\tensor h}}| = \left(|\gamma|^{(d-1)}\right)^h.
   \]

   Using \pref{obs:construction-projection}(b)--(c), the sets $\{ \Pi_{d-1}^{-1}(B; \bT^{(d)}_{\gamma^{\tensor n}}) : B \in \Pi_{d-1}(B_G(A,R)) \}$ form a collection
   of vertex-disjoint paths in $G$, and \pref{lem:inverse-diam} gives a lower bound on the diameter of every such path in $G$.
   It follows that for
   \begin{equation}\label{eq:Rprime}
      R' < \frac{1}{2}\left(\frac{1}{L_{\bT_{\gamma^{\tensor n}}^{(d)}}} - 1\right),
   \end{equation}
   there are at least $(|\gamma|^{(d-1)})^h$ vertex-disjoint paths originating in $B_G(A,R)$ and leaving $B_G(A,R')$.

   Finally, note that, by assumption, we have $\min(\gamma) \geq b$, and therefore $\min(\gamma^{\otimes n}) \geq b^n$,
   implying that 
   \[
      L_{\bT_{\gamma^{\tensor n}}^{(d)}} = b^{-n}.
   \]
   Since $\diam(G) \leq (d+1) b^n$ by \pref{lem:HabnDiam}, the constraint \eqref{eq:Rprime} is implied by
   \[
      R' < \frac{1}{2} \left(\frac{\diam(G)}{d+1} - 1\right).
   \]
   We may assume that $\diam(G) > 6(d+1)$, otherwise the statement of the lemma is vaccuous
   (since we may choose the constant $c$ sufficiently small depending on $d$), in which case this is implied by
   \[
      R'\leq \frac{\diam(G)}{3(d+1)},
   \]
   completing the proof.
\end{proof}

\section{Sphere-packing representations}
\label{sec:sphere-packing}

We will now prove \pref{lem:gamma-sphere-packing} and \pref{thm:sphere-packing-ghat}.
To this end, we require some regularity from our cube packings.
Say that two closed, axis-parallel cubes $A,B \subseteq \R^d$ are {\em neatly tangent} if:

\begin{enumerate}[(i)]
   \item $A$ and $B$ are interior-disjoint.
   \item If $A$ and $B$ intersect along $(d-1)$-dimensional faces $F_A \subseteq A$ and $F_B \subseteq B$,
      then either $F_A \subseteq A \cap B$ or $F_B \subseteq A \cap B$.
\end{enumerate}

A collection $\cC$ of closed, axis-parallel cubes is {\em a neat cube packing} if every pair $A \neq B \in \cC$
is either neatly tangent or else disjoint.  The {\em aspect ratio of the packing $\cC$} is defined by
\[
   \alpha(\cC) \seteq \max \left\{ \frac{\ell(A)}{\ell(B)} : A,B \in \cC, A \cap B \neq \emptyset \right\},
\]
where $\ell(A)$ denotes the sidelength of a cube $A$.
Say that a graph $G$ is {\em admits an $\alpha$-uniform neat cube packing in $\R^d$} if $G$
is the tangency graph of a neat cube packing $\cC$ with $\alpha(\cC) \leq \alpha$.

\begin{lemma}
   If $\gamma = \langle \gamma_1,\ldots,\gamma_b\rangle$ satisfies $\gamma_1=\gamma_b=b=\min(\gamma)$ and \eqref{eq:integral} holds,
   then the graphs $G(\bT_{\gamma}^{(d)})$ and $\hat{G}_{\gamma}^{(d)}$ admit an $\alpha$-uniform neat cube packing
   in $\R^d$ with $\alpha = \alpha_{\bT_{\gamma}^{(d)}}$.
\end{lemma}

\begin{proof}
   First take $\cC \seteq \tilde{\bT}_{\gamma}^{(d)}$, as defined in \pref{sec:infinite-graphs}.
   Under the integrality assumptions on $\gamma$, it holds that $\cC$
   is a neat packing, since one of the ratios $\gamma_{i+1}/\gamma_i$ or $\gamma_i/\gamma_{i+1}$ is an integer for every $1 \leq i < b$
(see \pref{fig:cube-packing} for an illustration).  Moreover, we have
\[
   \alpha(\cC) = \alpha_{\bT_{\gamma}^{(d)}} = \alpha_{\tilde{\bT}_{\gamma}^{(d)}}.
\]

For the second assertion, we take $\cC \seteq \bT_{\gamma^{\tensor \infty}}^{(d)}$ (as defined in \pref{sec:infinite-graphs}).
\pref{cor:alpha-power-bound} asserts that
\[
   \alpha_{\bT_{\gamma^{\otimes n}}^{(d)}} \leq \alpha_{\bT_{\gamma}^{(d)}}, \quad \forall n \geq 0,
\]
hence $\alpha_{\bT_{\gamma^{\otimes \infty}}^{(d)}} < \infty$, as desired.
\end{proof}

Given the preceding lemma, the next result suffices to prove \pref{lem:gamma-sphere-packing} and \pref{thm:sphere-packing-ghat}.

\begin{lemma}\label{lem:neat-cube-sphere-packing}
   If $G$ admits an $\alpha$-uniform neat cube packing in $\R^d$, then there are numbers $m\leq O(\alpha d)$ and $M \leq O(\alpha^2 d)$ such that
   the subdivision $[G]_m$ is $M$-uniformly sphere-packed in $\R^d$.
\end{lemma}

To prove this, we need a simple result on sphere packings that satisfy
prescribed tangencies.  
For a general closed axis-parallel cube $C \subseteq \R^d$ and $\e > 0$, we define $\partial C$ to 
denote the boundary of $C$ in $\R^d$, and we use $\|\cdot\|$ for the standard Euclidean distance.
For a point $x \in \R^d$ and a set $S \subseteq \R^d$, define $\dist_2(x,S) \seteq \inf \{ \|x-y\| : y \in S \}$.
Let us also define
\[
   \partial_{\e} C \seteq \{ x \in \partial C : \# \{ F \in \cF_C : \dist_2(x, F) > \e \ell(C) \} = 2d-1 \},
\]
where $\ell(C)$ is the sidelength of $C$, and $\cF_C$ is the collection of the $2d$ facets of $C$ (i.e., the $(d-1)$-dimensional faces of $C$).
These are the boundary points that lie on exactly once face of $C$ and are $\e \ell(C)$-far from every other.

\begin{lemma}\label{lem:star-incidence}
   For every $d \geq 2$ and $\e \in (0,1/2)$, the following holds.
   Consider a closed axis-parallel cube $C \subseteq \R^d$ of sidelength $\ell$, and a set of points 
   $P \subseteq \partial_{\e} C$ such that $\|x-y\| \geq \e \ell$ for every $x \neq y \in P$.
   Then there is a finite collection of interior-disjoint spheres $\cS$ contained in $C$ such that:
   \begin{enumerate}
      \item The radius of every sphere in $\cS$ is at least $\e \ell /(60 d)$.
      \item For every $p \in P$, there is a sphere $S(p) \in \cS$ tangent to $p$.
      \item The tangency graph of $\cS$ is the $m$-subdivision of a star graph whose leaves are the spheres $\{S(p) : p \in P\}$,
         with $m = 2 + \lceil \frac{6 d}{\e} \rceil$.
   \end{enumerate}
\end{lemma}

\begin{proof}
   By scaling and translation, it suffices to prove the lemma for $C = [0,1]^d$.
   Denote $c_0 \seteq (\frac12,\frac12,\ldots,\frac12)$,
   and define the sphere \[S_0 \seteq \{ x \in \R^d : \|x-c_0\| = 1/4 \}.\]
   For each $p \in P$, define $S(p)$ to be the unique sphere of radius $\e/8$ that is contained in $[0,1]^d$
   and such that $S(p) \cap \partial [0,1]^d = \{p\}$.  Such a sphere exists because $p \in \partial_{\e} [0,1]^d$.
   
   Let $c_p$ denote the center of $S(p)$, and let $[c_0,c_p]$ denote the line segment from $c_0$
   to $c_p$.  Define $z_p$ to be the point where $[c_0,c_p]$ intersects $S_0$, and define $z'_p$ to be the
   point where $[c_0,c_p]$ intersects $S(p)$.
   Let $[z_p,z'_p] \subseteq [c_0,c_p]$ denote the line segment connecting $z_p$ to $z'_p$.

   As $\|x-y\| \geq \e$ for all $x \neq y \in P$, it holds that $\|z'_x-z'_y\| \geq \e - 4 (\e/8) \geq \e/2$.
   Note that $[0,1]^d \subseteq B_2(c_0, \sqrt{d}/2)$, where the latter object is the Euclidean ball of radius $\sqrt{d}/2$
   about $c_0$.  Let $\tilde z_x$ denote the point where the line through $[c_0,c_p]$ intersects $\partial B_2(c_0,\sqrt{d}/2)$.
   Then we have $\|\tilde{z}_x - \tilde{z}_y\| \geq \|z'_x - z'_y\|$,
   and by similarity of the triangles defined by $\{c_0,\tilde{z}_x,\tilde{z}_y\}$ and $\{c_0,z'_x,z'_y\}$, it holds that
   \[
      \|z'_x - z'_y\| \geq \frac{\|\tilde{z}_x-\tilde{z}_y\|}{\sqrt{d}} \geq \frac{\e}{2\sqrt{d}}.
   \]
   It follows that
   \begin{equation}\label{eq:separated-segments}
      \min \{ \|a-b\| : a \in [z_x,z'_x], b \in [z_y,z'_y]\} \geq \|z_x-z_y\| \geq \frac{\e}{2\sqrt{d}}.
   \end{equation}
   Note also that for every $p \in P$,
   \begin{equation}\label{eq:max-len}
      \frac18 \leq \frac{1}{4} - \frac{\e}{4} \leq \|c_0 - c_p\| - \left(\frac{1}{4} + \frac{\e}{8}\right) \leq \left\|z_p-z'_p\right\| \leq \|c_0-c_p\| \leq \diam_2([0,1]^d) \leq \sqrt{d},
   \end{equation}
   where we have used $\e < 1/2$.

   Let $\gamma_p : [0, \|z_p-z'_p\|] \to [z_p,z'_p]$ be a parameterization of $[z_p,z'_p]$ by arclength.
   Define $k \seteq \lceil \frac{6 d}{\e}\rceil$ and $r_p \seteq \frac{\|z_p-z'_p\|}{2k}$, and let $\tilde{S}_p$
   be the collecion of interior-disjoint spheres of radius $r_p$ centered at the points
   \[
      \gamma_p(r_p), \gamma_p(2 r_p), \gamma_p(4 r_p), \ldots, \gamma_p(2(k-1) r_p).
   \]
   Note that the tangency graph of $\tilde{S}_p$ is a path, and that the first sphere is tangent to $S_0$,
   while the last is tangent to $S(p)$.

   By \eqref{eq:max-len}, for each $p \in P$, we have $r_P \in \left[ \frac{\e}{60 d}, \frac{\e}{6\sqrt{d}}\right]$.
   In particular, \eqref{eq:separated-segments} implies that if $S \in \tilde{S}_p$ and $S' \in \tilde{S}_{p'}$ for $p \neq p' \in P$, then
   $S$ and $S'$ are disjoint.

   Thus the collection of spheres
   \[
      \cS \seteq \{ S_0 \} \cup \{ S(p) : p \in P \} \cup \bigcup_{p \in P} \tilde{S}_p
   \]
   satisfies the conditions of the lemma.
\end{proof}

\begin{proof}[Proof of \pref{lem:neat-cube-sphere-packing}]
   Suppose that $\cC$ is a neat cube packing whose tangency graph is $G$ and such that $\alpha(\cC) \leq \alpha$.
   For every pair $A,B \in \cC$ with $A \cap B \neq \emptyset$, let $c(A,B)$ be the center of mass of $A \cap B$.

   Define the set of points $P_A \seteq \{ c(A,B) : B \in \cC, A \cap B \neq \emptyset \}$.  Since $\cC$ is a neat cube packing
   and we have $\ell(B) \geq \ell(A)/\alpha$, it follows that with $\e \seteq 1/(2\alpha)$, we have $P_A \subseteq \partial_{\e} A$, and
   \[
      \|x-y\| \geq \e \ell(A), \quad \forall x \neq y \in P_A.
   \]
   Therefore we can use \pref{lem:star-incidence} to replace each cube $A \in \cC$ by a corresponding collection $\cS_A$ of spheres
   whose tangency graph is the $m$-subdivision of a star, the leaves of which are tangent to the points in $P_A$.
   Note that \pref{lem:star-incidence} gives $m = 2 + \lceil 12 \alpha d \rceil$.

   Moreover, any two adjacent spheres have their ratio of radii contained in the interval $[M^{-1},M]$ for
   \[
      M \seteq \alpha \frac{60 d}{\e} = 120 \alpha^2 d.
   \]
   Thus if we define $\cS \seteq \bigcup_{A \in \cC} \cS_A$, then $[G]_{2m}$ is the tangency graph of $\cS$,
   and the corresponding sphere-packing is $M$-uniform.
\end{proof}

\subsection*{Acknowledgements}

This research was partially supported by NSF CCF-2007079 and a Simons Investigator Award.

\bibliographystyle{alpha}
\bibliography{diffusive}

\begin{thebibliography}{MTTV98}

\bibitem[Ang03]{Angel03}
O.~Angel.
\newblock Growth and percolation on the uniform infinite planar triangulation.
\newblock {\em Geom. Funct. Anal.}, 13(5):935--974, 2003.

\bibitem[BK02]{bk02}
Mario Bonk and Bruce Kleiner.
\newblock Quasisymmetric parametrizations of two-dimensional metric spheres.
\newblock {\em Invent. Math.}, 150(1):127--183, 2002.

\bibitem[BP11]{BP11}
Itai Benjamini and Panos Papasoglu.
\newblock Growth and isoperimetric profile of planar graphs.
\newblock {\em Proc. Amer. Math. Soc.}, 139(11):4105--4111, 2011.

\bibitem[BS01]{bs01}
Itai Benjamini and Oded Schramm.
\newblock Recurrence of distributional limits of finite planar graphs.
\newblock {\em Electron. J. Probab.}, 6:no. 23, 13 pp., 2001.

\bibitem[DG20]{DG20}
Jian Ding and Ewain Gwynne.
\newblock The {F}ractal {D}imension of {L}iouville {Q}uantum {G}ravity:
  {U}niversality, {M}onotonicity, and {B}ounds.
\newblock {\em Comm. Math. Phys.}, 374(3):1877--1934, 2020.

\bibitem[EL20]{EL20}
F.~Ebrahimnejad and J.~R. Lee.
\newblock On planar graphs of uniform polynomial growth.
\newblock To appear, {\em Probab. Theory Related Fields.} Preprint at
  \href{https://arxiv.org/abs/2005.03139}{arXiv:math/2005.03139}, 2020.

\bibitem[GHS20]{GHS20}
Ewain Gwynne, Nina Holden, and Xin Sun.
\newblock A mating-of-trees approach for graph distances in random planar maps.
\newblock {\em Probab. Theory Related Fields}, 177(3-4):1043--1102, 2020.

\bibitem[Kri05]{Krikun05}
M.~A. Krikun.
\newblock Uniform infinite planar triangulation and related time-reversed
  critical branching process.
\newblock {\em J. Math. Sci.}, 131:5520--5537, 2005.

\bibitem[LT79]{LT79}
Richard~J. Lipton and Robert~Endre Tarjan.
\newblock A separator theorem for planar graphs.
\newblock {\em SIAM J. Appl. Math.}, 36(2):177--189, 1979.

\bibitem[MTTV97]{MTTV97}
Gary~L. Miller, Shang-Hua Teng, William Thurston, and Stephen~A. Vavasis.
\newblock Separators for sphere-packings and nearest neighbor graphs.
\newblock {\em J. ACM}, 44(1):1--29, 1997.

\bibitem[MTTV98]{MTTV98}
Gary~L. Miller, Shang-Hua Teng, William Thurston, and Stephen~A. Vavasis.
\newblock Geometric separators for finite-element meshes.
\newblock {\em SIAM J. Sci. Comput.}, 19(2):364--386, 1998.

\bibitem[PRS94]{PRS94}
Serge Plotkin, Satish Rao, and Warren~D. Smith.
\newblock Shallow excluded minors and improved graph decompositions.
\newblock In {\em Proceedings of the Fifth Annual ACM-SIAM Symposium on
  Discrete Algorithms (Arlington, VA, 1994)}, pages 462--470, New York, 1994.
  ACM.

\bibitem[ST07]{spielman-teng}
Daniel~A. Spielman and Shang-Hua Teng.
\newblock Spectral partitioning works: Planar graphs and finite element meshes.
\newblock {\em Linear Algebra and its Applications: Special Issue in honor of
  Miroslav Fiedler}, 421(2--3):284--305, March 2007.

\bibitem[Ten98]{Teng98}
Shang-Hua Teng.
\newblock Combinatorial aspects of geometric graphs.
\newblock {\em Computational Geometry}, 9(4):277--287, 1998.

\end{thebibliography}

\end{document}